\documentclass{amsart}
\usepackage{graphicx}
\usepackage{url}
\usepackage{amssymb}
\usepackage[dvipdfmx]{hyperref}
\usepackage{amssymb}
\usepackage{url}
\usepackage{cite}

\newtheorem{theorem}{Theorem}[section]
\theoremstyle{plain}

\newtheorem{corollary}{Corollary}[section]

\newtheorem{lemma}{Lemma}[section]

\numberwithin{equation}{section}
\begin{document}
\title[ Non-isomorphic graphs with common degree sequences ]{ Non-isomorphic graphs with common degree sequences }
\author{Rikio Ichishima}
\address{Department of Sport and Physical Education, Faculty of Physical
Education, Kokushikan University, 7-3-1 Nagayama, Tama-shi, Tokyo 206-8515,
Japan}
\email{ichishim@kokushikan.ac.jp}
\author{Francesc A. Muntaner-Batle}
\address{Graph Theory and Applications Research Group, School of Electrical
Engineering and Computer Science, Faculty of Engineering and Built
Environment, The University of Newcastle, NSW 2308 Australia }
\email{famb1es@yahoo.es}
\date{March 12, 2023}

\begin{abstract}
For all positive even integers $n$, graphs of order $n$ with degree sequence 
\begin{equation*}
S_{n}:1,2,\dots,n/2,n/2,n/2+1,n/2+2,\dots,n-1
\end{equation*} 
naturally arose in the study of a labeling problem in \cite{IMO}.
This fact motivated the authors of the aforementioned paper to study these sequences and as a result of this study they proved that there is a unique graph of order $n$ realizing $S_{n}$ for every even integer $n$.
The main goal of this paper is to generalize this result.
\end{abstract}
\dedicatory{In memory of Susana Clara L\'{o}pez Masip}
\maketitle

\section{Introduction}
Unless stated otherwise, the graph-theoretical notation and terminology used here will follow Chartrand and Lesniak \cite{CL}. 
In particular, we assume that graphs considered in this paper are simple, that is, without loops or multiple edges.  
To indicate that a graph has \emph{vertex set} $V$ and \emph{edge set} $E$, 
we write $G=\left(V,E\right)$. 
To emphasize that $V$ and $E$ are the vertex set and edge set of a graph $G$, 
we will write $V$ as $V\left(G\right)$ and $E$ as $E\left(G\right)$.

The \emph{removal of a vertex} $v$ from a graph $G$ results in that subgraph $G-v$ of $G$ consisting of all vertices of $G$ except $v$ and all edges not incident with $v$. 
Thus, $G-v$ is the maximal subgraph of $G$ not containing $v$. 
On the other hand, if $v$ is not adjacent in $G$, \emph{the addition of vertex} $v$ results in the smallest supergraph $G+v$ of $G$ containing the vertex $v$ 
and all edges incident with $v$.
The \emph{union} $G\cong G_{1}\cup G_{2}$ has $V\left(G\right)=V\left(G_{1}\right) \cup V\left(G_{2}\right)$ 
and $E\left(G\right)=E\left(G_{1}\right) \cup E\left(G_{2}\right)$. 

The \emph{degree} of a vertex $v$ in a graph $G$ denoted by $\deg_{G} v$ is the number of edges incident with $v$. 
A sequence $s : d_{1}, d_{2},\dots , d_{n}$ of nonnegative integers is called a \emph{degree sequence} of a graph $G$ of order $n$ if the vertices of $G$ can be labeled $v_{1}, v_{2},\dots, v_{n}$ so that $\deg v_{i} = d_{i}$ for $1 \leq i\leq n$. 
Throughout this paper, we write the degree sequence of a graph as an increasing sequence. 
A finite sequence $s$ of nonnegative integers is \emph{graphical} if there exists some graph that realizes $s$,
that is, $s$ is a degree sequence of some graph.

The concepts of graph isomorphism and isomorphic graphs are also crucial for the development of this paper, 
and although they are very basic in graph theory, we introduce them as a matter of completeness. 
Let $G_{1}=\left(V_{1},E_{1}\right)$ and $G_{2}=\left(V_{2},E_{2}\right)$ be two graphs. 
They are \emph{isomorphic} (written $G_{1}\cong G_{2}$) 
if there exists a bijective function $\phi:V_{1} \rightarrow V_{2}$ such that $xy \in E_{1}$ 
if and only if $\phi\left(x\right)\phi\left(y\right) \in E_{2}$. In this case, the function $\phi$ is called an \emph{isomorphism} from $G_{1}$ to $G_{2}$.

The following two lemmas regarding isomorphism of graphs are very elementary and fundamental, 
but nevertheless, necessary for the proof of our main result of this paper. 
Hence, we state and prove them next.

\begin{lemma}
\label{basic1}
Let $G_{1}=\left(V_{1},E_{1}\right)$ and $G_{2}=\left(V_{2},E_{2}\right)$ be two graphs of order $n$
for which there exist unique vertices $v_{1}\in V_{1}$ and $v_{2}\in V_{2}$ such that 
\begin{equation*}
\deg_{G_{1}}v_{1}=\deg_{G_{2}}v_{2}=n-1 \text{.}
\end{equation*}
Then $G_{1}\cong G_{2}$ if and only if $G_{1}-v_{1}\cong G_{2}-v_{2}$.
\end{lemma}
\begin{proof}
First, assume that $G_{1}\cong G_{2}$. 
Then there exists an isomorphism $\phi:V_{1} \rightarrow V_{2}$.
Since $v_{i}$ ($i=1,2$) are the only vertices of $V_{i}$ with degree $n-1$ and each isomorphism preserves degrees, 
it follows that $\phi\left(v_{1}\right)=v_{2}$.
Thus, if we consider $G_{1}-v_{1}$ and $G_{2}-v_{2}$, 
it follows that the function $\phi^{\prime}:V_{1}\backslash\{v_{1} \} \rightarrow V_{2}\backslash\{v_{2}\}$ defined by
$\phi^{\prime}\left(a\right)=\phi\left(a\right)$ for all $a\in V_{1}\backslash\{v_{1}\}$ is well defined and bijective.
Furthermore, $ab\in E\backslash\{v_{1}x\mid x\in V_{1}\backslash\{v_{1}\}\}$ if and only if 
$\phi^{\prime}\left(a\right)\phi^{\prime}\left(b\right)\in E\backslash\{v_{2}x\mid x\in V_{2}\backslash\{v_{2}\}\}$.
This implies that $\phi^{\prime}:V_{1}\backslash\{v_{1} \} \rightarrow V_{2}\backslash\{v_{2}\}$ is an isomorphism and hence $G_{1}\cong G_{2}$.

Next, assume that $H_{1}=\left(V_{1}^{\prime},E_{1}^{\prime}\right)$ and $H_{2}=\left(V_{2}^{\prime},E_{2}^{\prime}\right)$
are two isomorphic graphs with an isomorphism $\phi:V_{1}^{\prime} \rightarrow V_{2}^{\prime}$.
Also, let $v_{1}$ and $v_{2}$ be two new vertices and consider two graphs $H_{1}+v_{1}$ and $H_{2}+v_{2}$. 
We show that  $H_{1}+v_{1}\cong H_{2}+v_{2}$.
To do this, consider the function
$\phi^{\prime}:V\left(H_{1}+v_{1} \right) \rightarrow V\left(H_{2}+v_{2} \right)$ defined by
\begin{equation*}
\phi^{\prime}\left( v\right) =\left\{ 
\begin{tabular}{ll}
$\phi\left( v\right) $ & if $v\in V_{1}^{\prime}$ \\ 
$v_{2}$ & if $v=v_{1}$.%
\end{tabular}
\right.
\end{equation*}
We will show that $\phi^{\prime}$ is an isomorphism from $H_{1}+v_{1}$ to $H_{2}+v_{2}$.
Since $\phi$ is an isomorphism from $H_{1}$ to $H_{2}$, 
it follows that $ab\in E\left(H_{1}+v_{1}\right)$ and $\{a,b\} \cap\{v_{1}\}=\emptyset$ 
if and only if $\phi^{\prime}\left(a\right) \phi^{\prime}\left(b\right)\in E\left(H_{2}+v_{2}\right)$.
On the other hand, if $av_{1} \in E\left(H_{1}+v_{1}\right)$ for all $a\in V_{1}^{\prime}$ 
and $bv_{2}\in E\left(H_{2}+v_{2}\right)$ 
for all $b\in V_{2}^{\prime}$, 
then $\phi^{\prime}\left(a\right) \phi^{\prime}\left(v_{1}\right)=\phi\left(a\right)v_{2}\in E\left(H_{2}+v_{2}\right)$.
This implies that $\phi^{\prime}$ is an isomorphism from $H_{1}+v_{1}$ to $H_{2}+v_{2}$ so that $H_{1}+v_{1}\cong H_{2}+v_{2}$.
\end{proof}

\begin{lemma}
\label{basic2}
Let $G_{1}=\left(V_{1},E_{1}\right)$ and $G_{2}=\left(V_{2},E_{2}\right)$ be two graphs. 
If $v_{1}$ and $v_{2}$ are two new vertices, 
then $G_{1}\cong G_{2}$ if and only if $G_{1}\cup v_{1}\cong G_{2}\cup v_{2}$.
\end{lemma}
\begin{proof}
First, assume that $G_{1}\cong G_{2}$.
Then there exists an isomorphism $\phi:V_{1} \rightarrow V_{2}$.
Now, consider the function $\phi^{\prime}:V_{1}\cup\{v_{1}\} \rightarrow V_{2}\cup\{v_{1}\}$ defined by
\begin{equation*}
\phi^{\prime}\left( v\right) =\left\{ 
\begin{tabular}{ll}
$\phi\left( v\right) $ & if $v\in V_{1}$ \\ 
$v_{2}$ & if $v=v_{1}$.%
\end{tabular}
\right.
\end{equation*}
Since no edge of the form $av_{1}$ exists in $G_{1}\cup v_{1}$ 
and  no edge of the form $bv_{2}$ exists in $G_{2}\cup v_{2}$,
it follows that $\phi^{\prime}$ is an isomorphism from $G_{1}\cup v_{1}$ to $G_{2}\cup v_{2}$ 
and hence $G_{1}\cup v_{1}\cong G_{2}\cup v_{2}$.

Next, assume that $G_{1}\cup v_{1}\cong G_{2}\cup v_{2}$. 
Then there exists an isomorphism $\phi:V_{1} \rightarrow V_{2}$.
Since the image under $\phi$ of any isolated vertex is an isolated vertex,
we may assume, without loss of generality, that $\phi\left(v_{1}\right)=v_{2}$.
This implies that the function $\phi^{\prime}:V_{1} \rightarrow V_{2}$ defined by 
$\phi\left(v\right)=v$ for all $v\in V_{1}$ is clearly well defined, bijective and an isomorphism from $G_{1}$ to $G_{2}$.
Therefore, $G_{1}\cong G_{2}$.
\end{proof}

\section{Main results}
With the information provided in the introduction, we are ready to present our main results.

Let $S_{0}: 0 \leqq a_{1} \leqq a_{2} \leqq \cdots \leqq a_{n-1} \leqq a_{n}$ be a graphical sequence. 
If we assume that there exist exactly $k$ ($k\geq 1$) graphs that realize $S_{0}$, then we have the following result.

\begin{theorem}
\label{main}
The sequences

$S_{0}^{\left(1\right)}:1,a_{1}+1,a_{2}+1,\dots,a_{n}+1,n+1;$

$S_{0}^{\left(2\right)}:1,2,a_{1}+2,a_{2}+2,\dots,a_{n}+2,n+2,n+3;$

$S_{0}^{\left(3\right)}:1,2,3,a_{1}+3,a_{2}+3,\dots,a_{n}+3,n+3,n+4, n+5;$
\begin{equation*}
\vdots
\end{equation*}

$S_{0}^{\left(i\right)}:1,2,3,\dots,i,a_{1}+i,a_{2}+i,\dots,a_{n}+i,n+i,n+i+1,\dots, n+2i-1$
\begin{equation*}
\vdots
\end{equation*}

\noindent are all graphical. Furthermore, there exist exactly $k$, $k\geq 1$, 
connected non-isomorphic graphs that realize each one of the sequences
$S_{0}^{\left(1\right)}, S_{0}^{\left(2\right)},\dots, S_{0}^{\left(i\right)},\dots$.

\end{theorem}
\begin{proof}
We start by showing that each sequence $S_{0}^{\left(i\right)}$ is graphical for any positive integer $i$.
To do this, 
we only need to take a graph that realizes $S_{0}$,
introduce two new vertices and join one of these two new vertices with all remaining vertices.
Hence, $S_{0}^{\left(1\right)}$ is graphical.
To obtain a graph that realizes $S_{0}^{\left(2\right)}$, we just need to take a graph that realizes 
$S_{0}^{\left(1\right)}$ and once again introduce two new vertices joining one of these new two vertices with all remaining vertices.
If we continue this process inductively, we obtain a graph that realizes $S_{0}^{\left(i\right)}$ for any positive integer $i$.

Now, observe that since each graph realizing $S_{0}^{\left(i\right)}$ ($i\geq1$) has a vertex which is adjacent to all the other vertices, 
it follows that all these graphs are connected.
Thus, it remains to show that each one of these sequences realizes exactly $k$ ($k\geq 1$) graphs.
To see this, let $S_{0}=S_{0}^{\left(0\right)}$ and proceed by induction on the super subscript $i$ of $S_{0}^{\left(i\right)}$ for $i\geq 0$.
First, observe that $S_{0}^{\left(0\right)}$ has the property that there exist exactly $k$ ($k\geq 1$) non-isomorphic graphs that realizes $S_{0}^{\left(0\right)}$ by assumption. 

Next, let $i=l$ ($l\geq 0)$ and assume that there exist exactly $k$ ($k\geq 1$) non-isomorphic graphs realizing $S_{0}^{\left(l\right)}$.
Consider the sequence 
\begin{equation*}
S_{0}^{\left(l+1\right)}: 1,2,\dots,l,l+1,a_{1}+l+1,a_{2}+l+1,\dots,a_{n}+l+1,n+l+1,n+l+2,\dots,n+2l+1 \text{.}
\end{equation*}
and let $ G_{0}^{\left(l+1\right)}$ be any graph that realizes $S_{0}^{\left(l+1\right)}$.
It is now clear that the vertex of degree $n+2l+1$ is adjacent to all other vertices of $V\left(G_{0}^{\left(l+1\right)}\right)$.
It is also true that if we eliminate this vertex, then we obtain a new graph with degree sequence $0, S_{0}^{\left(l\right)}$.
By inductive hypothesis, there exist exactly $k$ ($k\geq 1$) non-isomorphic graphs with degree sequence  $S_{0}^{\left(l\right)}$.
Then Lemma \ref{basic2} yields that there are exactly $k$ ($k\geq 1$) non-isomorphic graphs with degree sequence $0, S_{0}^{\left(l\right)}$, 
and Lemma \ref{basic1} implies that there are exactly $k$ ($k\geq 1$) non-isomorphic graphs realizing $S_{0}^{\left(l+1\right)}$. 
Therefore, the result follows.
\end{proof}

To conclude this section, notice that it is clear that the only graph that realizes the sequence $s:1,1$ is the complete graph $K_{2}$ of order $2$. 
From this observation together with Theorem \ref{main}, the next result found in \cite{IMO} follows as an immediate corollary.

\begin{corollary}
\label{main2}
For all positive integers $n$, there exists a unique graph of order $n$ that realizes the sequence $S_{n}:1,2,\dots,n/2.n/2,n/2+1,n/2+2,\dots,n-1$.
\end{corollary}

In summary, what we have proved in this paper is that if a degree sequence is realized by exactly $k$ ($k\geq 1$) non-isomorphic graphs of order $n$, 
then there exist infinitely many sequences that realize exactly $k$ ($k\geq 1$) non-isomorphic graphs of order $n$ each. 
Furthermore, all these graphs have the additional property that they are all connected.

\subsubsection*{$\emph{Acknowledgment}$}
The authors would like to dedicate this paper to Susana Clara L\'{o}pez Masip who passed away on December 26, 
2022 after a life dedicated to graph theory.

The authors are also gratefully indebted to Yukio Takahashi for his technical assistance.

\end{document}